\documentclass{amsart}
\usepackage[all]{xy}

\DeclareMathSymbol{\twoheadrightarrow}  {\mathrel}{AMSa}{"10}

\def\Z{{\mathbb Z}}

\def\End{\mathrm{End}}

\def\Hom{\mathrm{Hom}}

\def\I{{\mathcal I}}

\def\dim{\mathrm{dim}}
                     \def\Mat{\mathrm{Mat}}

\newtheorem{thm}{Theorem}[section]

\newtheorem{prop}[thm]{Proposition}
\theoremstyle{definition}

\newtheorem{rem}[thm]{Remark}

\newtheorem{rems}[thm]{Remarks}
            
\title[Quaternion trick]
{Abelian varieties, quaternion trick and endomorphisms}
\author[Yuri\ G.\ Zarhin]{Yuri\ G.\ Zarhin}
\address{Department of Mathematics, Pennsylvania State University,
University Park, PA 16802, USA}

 \email{zarhin\char`\@math.psu.edu}
\thanks{The author  was partially supported by Simons Foundation Collaboration grant   \# 585711. Part of this work was done during his stay in 2022 at the Max-Planck Institut f\"ur Mathematik (Bonn, Germany), whose hospitality and support are gratefully acknowledged.}

\begin{document}

\begin{abstract}
The quaternion trick is an explicit construction that associates to a polarized abelian variety $X$ a principal polarization of $(X\times X^t)^4$. The aim of this note is to show that this construction is compatible with endomorphisms of $X$ and $X^t$. See Theorem \ref{mainEq} for a precise statement.

\end{abstract}

\maketitle

\section{Introduction}
Throughout this paper, $K$ is a field. If   $X$ is an abelian variety  over $K$ then we write $\End(X)$ for the ring of
all $K$-endomorphisms of $X$. If $m$ is an integer then we write $m_X$ for the
multiplication by $m$ in $X$; in particular, $1_X$ is the identity map.
(Sometimes we will use notation $m$ instead of $m_X$.)

If $Y$ is an abelian variety  over $K$ then we write $\Hom(X,Y)$ for the group
of all $K$-homomorphisms $X\to Y$.

Let $X$ be an abelian variety over a field $K$ and let $X^t$ be its dual. 
If $u$ is an endomorphism of $X$ then we write $u^t$ for the dual endomorphism of $X^t$.  The corresponding map of the endomorphism rings
$$\End(X) \to \End(X^t),  \ u \mapsto u^t$$
is an {\sl antiisomorphism} of rings. If $m$ is a positive integer then  there are the natural ``diagonal'' ring embeddings
$$\Delta_{m,X}: \End(X) \hookrightarrow \Mat_m(\End(X))=\End(X^m),$$
$$\Delta_{m,X^t}: \End(X^t) \hookrightarrow \Mat_m(\End(X^t))=\End((X^t)^m)=\End((X^m)^t),$$
which send $1_X$ (resp. $1_{X^t}$) to the identity automorphism of $X^m$ (resp. of $(X^t)^m$).
Namely, if $u$ is an endomorphism of $X$ (resp. of $X^t$) then
$\Delta_m(u)$ sends $(x_1,\dots x_m)$ to $(ux_1, \dots, ux_m)$ for all $(x_1,  \dots, x_m)$ in $X^m$ (resp. in $(X^t)^m$).
Clearly, the subring
$\Delta_{m,X}(\End(X))$ of $\End(X^m)$ commutes with the subring
$$\Mat_m(\Z)\subset \Mat_m(\End(X))=\End(X^m).$$

Let $\lambda: X \to X^t$ be a {\sl polarization} on $X$ that is defined over $K$. 
For every positive integer $m$ we consider 
the polarization
\begin{equation}
\label{polM}
\lambda^m: X^m \to (X^m)^t= (X^t)^m, \ (x_1, \dots , x_m)\mapsto (\lambda(x_1), \dots ,
\lambda(x_m))
\end{equation}
of $X^m$ that is also defined over $K$.
 We have
 \begin{equation}
 \label{degLambda}
\dim(X^m)=m\cdot \dim(X), \ \deg(\lambda^m)=\deg(\lambda)^m,
\ \ker(\lambda^m)={\ker(\lambda)}^m \subset X^m.
\end{equation}

The aim of this note is to prove the following assertion.

\begin{thm}
\label{mainEq}
Let $\lambda: X \to X^t$ be a polarization on $X$ that is defined over $K$.
Let $O$ be an associative ring with $1$
endowed with an involutive antiautomorphism
$$O \to O, \ e \mapsto e^{*}.$$
Suppose that we are given a ring embedding
$\iota: O \to \End(X)$ 
that sends $1$ to $1_X$ and such that in $\Hom(X,X^t)$
\begin{equation}
\label{lambdaO}
\lambda \circ \iota(e)= \iota(e^{*})^t \circ\lambda, \ \forall e \in O.
\end{equation}

Then:
\begin{itemize}
\item[(i)] The map
\begin{equation}
\label{iota8}
\iota^{*}: O \to \End(X^t),   e \mapsto \left(\iota(e^{*})\right)^t
\end{equation}
is a ring embedding
that sends $1$ to $1_{X^t}$.
\item[(ii)]  
Let us consider the ring embeding
\begin{equation}
\label{embed4}
\kappa_4=\Delta_{4,X^t}\circ \iota^{*}\oplus \Delta_{4,X}\circ \iota: O \hookrightarrow
\End((X^t)^4)\oplus \End(X^4) \subset \End\left((X^t)^4\times X^4\right)=\End((X\times X^t)^4),
\end{equation}
$$e \mapsto \left(\Delta_{4,X^t}(\iota^{*}(e)),\Delta_{4,X}(\iota(e)\right)\in
\End((X^t)^4)\oplus \End(X^4)\subset \End\left((X^t)^4\times X^4\right).$$
Then there exists a principal polarization $\mu$ on $X^4 \times (X^t)^4 =(X \times X^t)^4$ that is defined over $K$ and enjoys the following properties.

\begin{equation}
\label{kappa4O}
\mu\circ \kappa_4(u)= \kappa_4(u^{*})^t \circ \mu \ \forall u \in O.
\end{equation}
\end{itemize}
\end{thm}

\begin{rem}
Clearly, $\left(1_X\right)^t=1_{X^t}$. Since both  maps
$$\End(X) \to \End(X^t), \ u \to u^t \  \text{ and } \  O \to O, \ e \mapsto e^{*}$$
are ring antiisomorphisms, and $\iota: O \to \End(X)$ is a ring embedding that sends $1$ to $1_X$, the composition
$$O \to \End(X^t), \  e \mapsto \left(\iota(e^{*})\right)^t$$
is obviously a ring embedding that sends $1$ to $1_{X^t}$, which proves Theorem \ref{mainEq}(i). The rest of the paper is devoted 
to the proof of Theorem \ref{mainEq}(ii).
\end{rem}

\begin{rems}
\label{lambdaM}
We keep the notation and assumptions of Theorem \ref{mainEq}. 

\begin{itemize}
\item
Formula \eqref{lambdaO} implies that for every positive integer $m$
\begin{equation}
\label{lambdaOM}
\lambda^m \circ \Delta_{m,X}(\iota(e))= (\Delta_{m,X^t}(\iota(e^{*}))^t \circ\lambda^m \ \forall e \in O.
\end{equation}
\item
It is well known \cite{MumfordAV} that the additive group of $\End(X)$ is a  free $\Z$-module of finite rank.  Since $\iota$ is an embedding, the additive group of $O$ is also  a free $\Z$-module of finite rank. 
\end{itemize}
\end{rems}

\begin{rem}
When $O=\Z$, the assertion of Theorem \ref{mainEq} is a well known {\sl quaternion trick} \cite[Lemma 2.5]{ZarhinMatZam},
\cite[Sect. 5]{ZarhinInv85}, \cite[Sect.  1.13 and 7]{ZarhinG}. (See \cite[Ch. IX, Sect. 1]{MB} where
Deligne's proof is given.) 
\end{rem}

This note may be viewed as a natural continuation of \cite{ZarhinG}. In particular, we freely use a (more or less) standard notation from  \cite{ZarhinG}. 

The paper is organized as follows.
Section \ref{polIs} contains auxiliary results that deal with interrelations between polarizations, isogenies and endomorphisms of abelian varieties.  In Section \ref{BC} we compare the situation over $K$ and over its fixed algebraic closure $\bar{K}$.
We prove Theorem \ref{mainEq} (ii) in Section \ref{ProofM}.

{\bf Acknowledgements}. I am grateful to  Yujie Xu for an interesting question; this note is a result of my attempts to answer it. My special thanks go to both referees, whose comments helped to improve the exposition.

\section{Polarizations and isogenies}
\label{polIs}

\begin{prop}
\label{polDescent}
Let $(Y, \lambda)$ be a polarized abelian variety over $K$. Let $\pi: Y \to Z$ be a $K$-isogeny of abelian varieties over $K$.  Suppose that there exists a polarization $\mu: Z \to Z^t$ that is defined over $K$ and such that  
\begin{equation}
\label{lambdaMu}
\lambda=\pi^t \circ \mu \circ \pi .
\end{equation}
Let $D$ be an associative ring with $1$
endowed with an involutive  antiautomorphism
$$D \to D, \ e \mapsto e^{*}.$$
Suppose that we are given an injective ring embedding
$j: D \to \End(Y)$ 
that sends $1$ to $1_X$ and such that in $\Hom(Y,Y^t)$
\begin{equation}
\label{lambdaJ}
\lambda \circ j(e)= j(e^{*})^t \circ\lambda \ \forall e \in D.
\end{equation}
Suppose that there exists a ring embedding
 $j_{\pi}:D \hookrightarrow \End(Z)$  that sends $1$ to $1_Z$ and such that
 \begin{equation}
\label{Desc}
 j_{\pi}(e)\circ \pi=\pi \circ j(e)  \ \forall e \in D.
 \end{equation}
 Then
\begin{equation}
\label{muJpi}
\mu \circ j_{\pi}(e)= j_{\pi}(e^{*})^t \circ\mu \ \forall e \in D.
\end{equation}
\end{prop}

\begin{proof}
Let $u \in D$.
Plugging in  formula  \eqref{lambdaMu} for $\lambda$ into \eqref{lambdaJ}, we obtain
$$\pi^t \circ\mu \circ\pi \circ j(e)= j(e^{*})^t \circ \pi^t\circ  \mu \circ \pi.$$
Formula  \eqref{Desc} allows us to replace $\pi \circ j(e)$ by $j_{\pi}(e)\circ \pi$ and get
$$\pi^t \circ  \mu \circ  j_{\pi}(e)\circ \pi= j(e^{*})^t \circ \pi^t \circ \mu\circ \pi.$$
Dividing both sides by isogeny $\pi$ from the right, we get
$$\pi^t \circ\mu\circ  j_{\pi}(e)= j(e^{*})^t \circ \pi^t \circ  \mu .$$
Taking into account that $j(e^{*})^t \circ \pi^t =(\pi \circ j(e^{*}))^t$, we get
\begin{equation}
\label{interm}
\pi^t \circ \mu \circ  j_{\pi}(e)= (\pi \circ j(e^{*}))^t \circ \mu.
\end{equation}
Applying \eqref{Desc} to $e^{*}$ instead of $e$, we get 
$j_{\pi}(e^{*})\circ \pi=\pi \circ j(e^{*})$. Combining this with \eqref{interm}, we obtain
$$ \pi^t \circ \mu \circ  j_{\pi}(e)= (j_{\pi}(e^{*})\circ \pi)^t \circ \mu =
 \pi^t\circ j_{\pi}(e^{*})^t \circ \mu$$
 and therefore
 $$ \pi^t \circ \mu \circ  j_{\pi}(e)= \pi^t\circ j_{\pi}(e^{*})^t \circ \mu.$$
 Dividing both sides by isogeny $\pi^t$ from the left, we get
 $$\mu  \circ  j_{\pi}(e) =
 j_{\pi}(e^{*})^t \circ \mu,$$
 which proves the desired formula \eqref{muJpi}.
\end{proof}

\section{Base change}
\label{BC}

In what follows,  $\bar{K}$ stands for a fixed algebraic closure of $K$. If $X$ (resp. $W$) is an algebraic variety (resp. group scheme) over $K$ then we write $\bar{X}$ (resp. $\bar{W}$) for the corresponding algebraic variety (resp. group scheme) over $\bar{K}$. Similarly, if $f$ is a morphism of $K$-varieties (resp. group schemes) then we write $\bar{f}$ for the corresponding morphism of algebraic varieties (resp. group schemes) over $\bar{K}$. In particular, if $X$ is an abelian variety with  $K$-polarization $\lambda: X \to X^t$  then
\begin{equation}
\label{barL}
\bar{\lambda}: \bar{X}\to \overline{X^t}=\bar{X}^t
\end{equation}
is a polarization of $\bar{X}$,  and
\begin{equation}
\label{barD}
\deg(\bar{\lambda})=\deg(\lambda), \  \ker(\bar{\lambda})=\overline{\ker(\lambda)}, 
\end{equation}
$$\overline{X^m}=\bar{X}^m, \ \overline{X^m}^t=(\overline{X^t})^m, \ \overline{\lambda^m}=\bar{\lambda}^m$$
for all  positive integers  $m.$

If $W$ is a finite commutative group scheme then $\bar{W}$ is a finite commutative group scheme over $\bar{K}$ and the {\sl orders} of $W$ and $\bar{W}$ coincide.  We have
\begin{equation}
\label{Wpower}
\bar{W}^m=\overline{W^m} \ \ \forall m.
\end{equation}
In addition, if $d$ is the order of $W$ then the orders  of $\bar{W}^m$  and $W^m$ both equal $d^m$. (See  \cite{DG,Shatz,HM,OS,Pink,ZarhinG} for a furher discussion of commutative finite group schemes over fields.)

\section{Quaternion trick}
\label{ProofM}
In what follows, we freely use the notation and assertions of Section \ref{BC}.

\begin{proof}[Proof of Theorem \ref{mainEq}]

Let us put $g:=\dim(X)$. We may assume that $g\ge 1$.
Recall that $\lambda$ is an isogeny and therefore $\ker(\lambda)$ is a finite group subscheme  in $X$. Let $n:=\deg(\lambda)$. Then  $\ker(\lambda)$ has order $n$ and therefore is killed by multiplication by $n$,
see \cite{OT}.

Choose a quadruple of integers $a,b,c,d$ such that $$
s:=a^2+b^2+c^2+d^2$$ is congruent to $-1$ modulo $n$. (In particular, $s \ne 0$.) We denote by
$\I$ the ``quaternion"
$$\I=
\begin{pmatrix}
a & -b & -c & -d\\
b & \ a & \ d &\  -c\\
c & -d & \ a & \ b\\
d & \ c & -b & \ a\\
\end{pmatrix}
\in \Mat_4(\Z)\subset \Mat_4(\End(X))=\End(X^4).$$ 

Following \cite[pp. 330-331]{ZarhinG}, let us consider a finite group subscheme  $$V\subset \ker(\lambda^4)\times \ker(\lambda^4)
\subset X^4 \times X^4=X^8$$ that is the graph of
$$\I: \ker(\lambda^4) \to \ker(\lambda^4).$$
In particular, $V$ and $\ker(\lambda^4)$ are isomorphic finite group schemes over $K$ and therefore have the same order, namely, $n^4$.
Clearly, 
$$\bar{V}\subset \ker(\bar{\lambda}^4)\times \ker(\bar{\lambda}^4)=\ker(\bar{\lambda}^8),$$
 the orders of  isomorphic finite group $\bar{K}$-schemes $\bar{V}$ and $ \ker(\bar{\lambda}^4)$ coincide and also  equal $n^4$.  It is  checked in  \cite[pp. 330--331]{ZarhinG}  that $\bar{V}$ is an {\sl isotropic} finite group subscheme in $\ker(\bar{\lambda}^8)$ with respect to the {\sl Riemann form} \cite[Sect. 23]{MumfordAV}
$$e_{\bar{\lambda}^8}: \ker(\bar{\lambda}^8)\times \ker(\bar{\lambda}^8) \to \mathbb{G}_{\mathrm{m}}$$
attached to the polarization 
$$\bar{\lambda}^8:\bar{X}^8 \to \left(\bar{X}^8\right)^t$$
of $\bar{X}^8=\overline{X^8}$.  (Here $\mathbb{G}_{\mathrm{m}}$ is the {\sl multiplicative} group scheme over $\bar{K}$.) Since the order of   $\bar{V}$ is $n^4=\sqrt{n^8}$, it is is the {\sl square root} of the order of $ \ker(\bar{\lambda}^8)$. This means that  $\bar{V}$ is a {\sl maximal isotropic} finite group subscheme of $\ker(\bar{\lambda}^8)$.

Let us consider a $K$-morphism of $8g$-dimensional abelian varieties
\begin{equation}
\label{piX8}
\pi: X^8=X^4 \times X^4\to (X^t)^4 \times X^4, \ (x_4,y_4) 
\mapsto (\lambda^4(x_4), \I(x_4)-y_4) \ \forall x_4,y_4 \in X^4.
\end{equation}
Clearly, $\ker(\pi)=V$, which is a finite group scheme, hence $\pi$ is an {\sl isogeny}
and  
$$X^4\times (X^t)^4 \cong X^8/V.$$
  In light of {\sl descent theory} \cite[Sect. 1.13]{ZarhinG} (applied to $X^8, \lambda^8, (X^t)^4\times X^4$ instead of $X,\lambda,Y$), 
the {\sl maximal isotropy} of $\bar{V}$ implies  that there exists a {\sl principal polarization} 
$$\mu:(X^t)^4\times X^4 \to  \left((X^t)^4\times X^4\right)^t$$
on  $(X^t)^4\times X^4$ that is defined over $K$ and such that
 \begin{equation}
 \label{lambda8mu}
\pi^t \circ \mu \circ \pi=\lambda^8: X^8 \to (X^8)^t=(X^t)^8.
\end{equation}
Let us consider the ring embeddings
$$j=\Delta_{8,X}\circ \iota: O \to \End(X^8)=\Mat_2(\End(X^4)), \ 
e \mapsto  
\begin{pmatrix}
\Delta_{4,X}(\iota(e)) & 0\\
0 & \Delta_{4,X}(\iota(e))\\
\end{pmatrix} $$
and
$$j_{\pi}= \kappa_4: O \to \Mat_4(\End(X^t)\oplus \Mat_4(\End(X))=\End\left((X^t)^4\right)\oplus \End(X^4)\subset \End\left((X^t)^4\times X^4\right),$$
$$e \mapsto \left(\Delta_{4,X^t}(\iota^{*}(e)), \Delta_4(\iota(e))\right)=
\begin{pmatrix}
\Delta_{4,X^t}(\iota^{*}(e)) & 0\\
0 & \Delta_{4,X}(\iota(e))\\
\end{pmatrix}
.$$
Let us  put $Y=X^8$, $Z= (X^t)^4\times X^4$,  $D=O$,  $m=4$ and check
that  $j$ and $j_{\pi}$  enjoy property \eqref{Desc}.
First, notice that the matrix
$$\I \in \Mat_4(\Z)\subset \Mat_4(\End(X))=\End(X^4)$$
{\sl commutes} with  the ``scalar'' matrix 
$$\Delta_{4,X}(\iota(e))=
\begin{pmatrix}
\iota(e) & 0 & 0& 0\\
0& \iota(e) &0 & 0\\
0&0& \iota(e) & 0\\
0 & 0& 0& \iota(e)\\
\end{pmatrix}
\in \Mat_4\left(\End(X)\right)=\End\left(X^4\right),$$
i.e.
\begin{equation}
\label{Icommute}
\I \circ \Delta_{4,X}(\iota(e))=\Delta_{4,X}(\iota(e))\circ \I.
\end{equation}
Second, plugging $m=4$ in \eqref{lambdaOM}, we get
\begin{equation}
\label{lambdaOM4}
\lambda^4\circ \Delta_{4,X}(\iota(e))=\Delta_{4,X^t}(\iota^{*}(e))\circ \lambda^4 \ \forall e\in O.
\end{equation}
Third, if 
$$e \in O, \  (x_4,y_4) \in X^4(\bar{K})\times X^4(\bar{K})=X^8(\bar{K})$$
then
$$\pi\circ j(e) (x_4,y_4)=\pi \left(\Delta_{4,X}(\iota(e))(x_4), \Delta_{4,X}(\iota(e))(y_4)\right)=$$
$$\left(\lambda^4\circ \Delta_{4,X}(\iota(e))(x_4), \I\circ \Delta_{4,X}(\iota(e))(x_4)-\Delta_{4,X}(\iota(e))(y_4)\right).$$
Taking into account  equalities \eqref{lambdaOM4}  and   \eqref{Icommute}, we obtain that
$$\pi\circ j(e) (x_4,y_4)=\left(\Delta_{4,X^t}(\iota^{*}(e))\circ \lambda^4(x_4),  \Delta_{4,X}(\iota(e))\I(x_4)-\Delta_{4,X}(\iota(e))(y_4)\right)$$
$$=\begin{pmatrix}
\Delta_{4,X^t}(\iota^{*}(e)) & 0\\
0 & \Delta_{4,X}(\iota(e))\\
\end{pmatrix}
\left(\lambda^4(x_4), \I(x_4)-y_4\right)=
\begin{pmatrix}
\Delta_{4,X^t}(\iota^{*}(e)) & 0\\
0 & \Delta_{4,X}(\iota(e))\\
\end{pmatrix}
\circ \pi \left(x_4,y_4\right)$$
$$=j_{\pi}(e) \circ \pi \left(x_4,y_4\right).$$
This means that
$$ j_{\pi}(e)\circ \pi=\pi \circ j(e)  \ \forall e \in O.$$
  Now the desired result follows from \eqref{lambda8mu} combined with Proposition \ref{polDescent} applied to  $\lambda^8$ (instead of  $\lambda$).
\end{proof}


\begin{thebibliography}{99}



\bibitem{DG} M. Demazure, P. Gabriel, Groupes alg\'ebriques, Tome
I. North Holland, Amsterdam 1970.



\bibitem{HM} R. Hoobler and A. Magid, {\em Finite group schemes over
fields}. Proc. Amer. Math. Soc. {\bf 33} (1972), 310--312.









 



\bibitem{MB} L. Moret-Bailly, Pinceaux de vari\'et\'es ab\'eliennes, Ast\'erisque, vol. {\bf 129} (1985).




\bibitem{MumfordAV} D. Mumford,  Abelian varieties, 2nd edition. Oxford University Press, 1974.



\bibitem{OS} F. Oort and J.R. Strooker, {\em The category of
finite groups over a field}. Indag. Math. {\bf 29} (1967),
163--169.

\bibitem{OT} F. Oort and J. Tate, {\em Group schemes of prime order}.
 Ann. Sci. \'Ecole Norm. Sup. $4^e$ s\'erie, tome {\bf 3} (1970), 1--21.


\bibitem{Pink} R. Pink, {\em Finite group schemes}. Lecture course in WS 2004/05
 ETH Z\"urich,
 http://www.math.ethz.ch/~pink/ftp/FGS/CompleteNotes.pdf .







\bibitem{Shatz} S.S. Shatz, {\sl Group schemes, Formal groups and}
$p$-{\sl divisible groups}. Chapter III in: Arithmetic Geometry
(G. Cornell, J.H. Silverman, eds.). Springer-Verlag, New York
1986.











\bibitem{ZarhinMatZam} Yu. G. Zarhin, {\em Endomorphisms of abelian varieties and
points of finite order in characteristic} $P$. Mat. Zametki, {\bf
21} (1977), 737-744;  Mathematical Notes {\bf 21} (1978) 415--419.



 \bibitem{ZarhinInv85} Yu. G. Zarhin, {\em A finiteness theorem for unpolarized Abelian varieties
 over number fields with prescribed places of bad reduction}.
 Invent. Math.  {\bf 79} (1985), 309--321.


\bibitem{ZarhinG} Yu. G. Zarhin, {\em Homomorphisms of abelian varieties over finite fields},
 pp.  315--343. In: Higher-dimensional geometry over finite fields, (D.
Kaledin, Yu. Tschinkel, eds.), IOS Press, Amsterdam, 2008; arXiv:0711.1615v5 [math.AG].


\end{thebibliography}
\end{document}